\documentclass[11pt,a4paper]{amsart}
\usepackage{amsfonts}
\usepackage{amsthm}
\usepackage{amssymb, amsmath}
\usepackage{amscd}
\usepackage[latin2]{inputenc}
\usepackage{t1enc}
\usepackage[mathscr]{eucal}
\usepackage{indentfirst}
\usepackage{graphicx}
\usepackage{graphics}
\usepackage{pict2e}

\theoremstyle{plain}
\newtheorem{Th}{Theorem}

\newtheorem{Cor}[Th]{Corollary}

\newtheorem{Pro}[Th]{Proposition}
 \theoremstyle{definition}

\newtheorem{?}[Th]{Problem}
\newcommand{\R}{\mathbb{R}}

\newcommand{\PP}{\mathbb{P}}\newcommand{\D}{\mathbb{D}}\newcommand{\B}{\mathbb{B}}

\begin{document}

\title[]
{On endomorphisms of the Einstein gyrogroup in arbitrary dimension}

\author[P. E. Frenkel]{P\'eter E. Frenkel}

\address{E\"{o}tv\"{o}s Lor\'{a}nd University \\ Department of Algebra and Number Theory
 \\ H-1117 Budapest
 \\ P\'{a}zm\'{a}ny P\'{e}ter s\'{e}t\'{a}ny 1/C \\ Hungary \\ and R\'enyi Institute of Mathematics, Hungarian Academy of Sciences \\ 13-15 Re\'altanoda utca \\ H-1053 Budapest}
\email{frenkel.peter@renyi.mta.hu}

\thanks{Research partially supported by ERC Consolidator Grant 648017 and by Hungarian National Foundation for Scientific Research (OTKA), grant no. K109684.}

 \subjclass[2010]{}

 \keywords{}

\begin{abstract}We determine the automorphisms and the continuous
  endomorphisms of the Einstein gyrogroup in arbitrary dimension.  This generalizes a recent result
  of L.\ Moln\'ar and D.\ Virosztek, who have determined the continuous
  endomorphisms in the three-dimensional case.
\end{abstract}

\maketitle

\section{The Einstein gyrogroup}

The $n$-dimensional Einstein gyrogroup is the open unit ball $\B^n$ in $\R^n$,
endowed with the binary operation of velocity addition from the special theory
of relativity (with the speed of light taken to be 1):
$$u\oplus
v=\frac1{1+(u,v)}\left(u+\sqrt{1-|u|^2}\cdot
  v+\frac{(u,v)}{1+\sqrt{1-|u|^2}}\cdot u\right)
.$$
Here $(u,v)$ is the inner product of $u$ and $v$, and $|u|=\sqrt {(u,u)}$ is
the usual Euclidean norm.

Note that $|u\oplus v|<1$ if $|u|<1$ and $|v|<1$, so $\left(\B^n, \oplus\right)$ is  an
algebraic structure. It satisfies certain axioms that make it a  gyrogroup \cite{U}. We
shall not need all of the axioms, but let us observe that $u\oplus 0=0\oplus u=u$ and $u\oplus
(-u)=0$ for all $u\in \B^n$. The operation $\oplus$ is not associative, but
$(-u)\oplus (u \oplus v)=v$ holds for all $u$
and $v$ in $\B^n$.

The Einstein gyrogroup is closely related to hyperbolic geometry. If we think
of $\B^n$ as the Cayley--Klein--Beltrami model of hyperbolic $n$-space, then
the map $v\mapsto u\oplus v$ is an isometry of hyperbolic $n$-space for any
fixed $u$. When $u\ne 0$, this isometry maps the halfline starting at 0 and
passing through $u$ (henceforth referred to as halfline $0u$) onto its sub-halfline starting at $u$.

This implies a well-known fact about commutativity in the Einstein gyrogroup:

\begin{Pro}\label{comm} Let $u, v\in\B^n$. Then the  equality $u\oplus v=v\oplus u$ holds if and only if $u$ and $v$ are linearly dependent
(in the usual sense of vector algebra in $\R^n$).
\end{Pro}

\begin{proof}
If $u$ and $v$ are linearly dependent, then they belong to a diameter of the
ball $\B^n$. This diameter represents a line $L$ in hyperbolic space. The  $L\to L$ maps
$w\mapsto u\oplus w$ and $w\mapsto v\oplus w$ are translations of $L$,
so they commute. Hence, $u\oplus v=u\oplus (v\oplus 0)=v\oplus (u\oplus
0)=v\oplus u$ as claimed.

If $u$ and $v$ are linearly independent, then they span a two-dimensional
plane, which intersects $\B^n$ in a  disc. This disc represents a hyperbolic
plane. Hyperbolic isometries preserve angles. Thus, the halfline
$u(u\oplus v)$  forms the same angle with the halfline $0u$ as $0v$
does. Hence, $$\angle 0u(u\oplus v)=\pi- \angle u0v.$$  Similarly,
 $$\angle 0v(v\oplus u)=\pi- \angle u0v.$$
If, by way of contradiction,  we have $u\oplus v=v\oplus u=w$, then corresponding sides of the triangles
$u0v$ and $vwu$ have equal length, making the two triangles congruent and
implying $$\angle u0v=\angle vwu.$$ But then the four angles of the
quadrilateral $u0vw$ sum to $2\pi$, which is impossible in the hyperbolic plane.
\end{proof}

\begin{Cor}\label{kollin} The points $x,y,z\in \B^n$ are collinear if and only if
\begin{equation}\label{coll} ((-x)\oplus y)\oplus  ((-x)\oplus z)=((-x)\oplus
  z)\oplus ((-x)\oplus y).\end{equation}\end{Cor}

\begin{proof}
In the Cayley--Klein--Beltrami model, lines of hyperbolic space are
represented by chords of the ball $\B^n$. Thus, $x$, $y$, $z$ are collinear in
the ordinary sense of Euclidean geometry if and only if they are collinear as
points of hyperbolic space.

The map $w\mapsto (-x)\oplus w$ is an isometry of hyperbolic space, so it
preserves collinearity. Thus $x$, $y$ and $z$ are collinear if and only if
$0$, $(-x)\oplus y$ and $(-x)\oplus z$ are. The claim now follows from
Proposition~\ref{comm}.
\end{proof}

\section{Endomorphisms and
automorphisms}  An endomorphism of the $n$-dimensional Einstein gyrogroup is a
map $f:\B^n\to \B^n$ such that \begin{equation}\label{endo}f(u\oplus v)=f(u)\oplus f(v)\end{equation} for all $u, v\in
\B^n$. An automorphism is a bijective endomorphism.

Note that any endomorphism $f$ satisfies $f(0)=0$.  Indeed, $0\oplus 0=0$, whence
$f(0)\oplus f(0)=f(0)=f(0)\oplus 0$.  But $v\mapsto f(0)\oplus v$ is
bijective, so $f(0)=0$.

When $n=1$, the Einstein gyrogroup is a group. It is isomorphic to the
additive group $(\R, +)$ of real numbers. Endomorphisms of this group have been
extensively studied, they go under the name of additive functions. Most of
them are non-continuous. Moreover, most of the automorphisms of $(\R, +)$ are
also non-continuous. In fact, the continuous endomorphisms are precisely the linear
functions $x\mapsto ax:\R\to\R$ with fixed $a\in \R$, and there are many further
automorphisms, let alone endomorphisms.

Henceforth, we assume $n\ge 2$.

\begin{Th}\label{aut} For $n\ge 2$, automorphisms of the Einstein gyrogroup $(\B^n,
  \oplus)$ are precisely the restrictions to $\B^n$ of the orthogonal
  transformations of $\R^n$.
\end{Th}

\begin{proof}
Orthogonal transformations of $\R^n$ preserve the inner product and therefore
the Euclidean norm, so they map $\B^n$ bijectively onto itself and satisfy
\eqref{endo} for all $u$ and $v$.

Conversely, if $f$ is an automorphism,
then so is its inverse $f^{-1}$.
By Corollary~\ref{kollin}, both $f$ and $f^{-1}$ map collinear points to
collinear points. I.e., $f$ --- as a self-map of hyperbolic space ---  maps
any line onto a  line. In other words, $f$ is a collineation of hyperbolic
space.  By a well-known result sometimes referred to as the fundamental
theorem of hyperbolic geometry \cite{J, LW, S, Y}, any collineation is an
isometry  for $n\ge 2$. So $f$ is an isometry. 
It is well known that in the Cayley--Klein--Beltrami model, any isometry of hyperbolic $n$-space fixing 0 is represented  by the restriction of an
orthogonal transformation.
\end{proof}

We now turn to endomorphisms. We urge the reader to solve
\begin{?}
For $n\ge 2$, is every endomorphism of the Einstein gyrogroup continuous?
\end{?}

Meanwhile, we wish to classify continuous endomorphisms. For $n=3$, which is the most
relevant to physics, this was done by L.\ Moln\'ar and D.\ Virosztek
\cite{MV}, while the general case was posed by them as an open problem.
Their result relies on a  chain of reinterpretations of $(\B^3,\oplus)$. The first
step in the chain is an observation of S.\ Kim \cite{K}: $(\B^3,\oplus)$ is
bicontinuously isomorphic to $(\D, \odot)$, where $\D$ is the set of 2-square
regular density matrices and $A\odot B$ is $\sqrt AB\sqrt A$ divided by its trace.
Moln\'ar and Virosztek show that this in turn is bicontinuously isomorphic to
$(\PP_2^1, \boxdot)$, where $\PP_2^1$ is the set of 2-square positive definite
matrices with determinant 1, and $A\boxdot B=\sqrt AB\sqrt A$. Then they invoke a  result from their previous paper \cite[Theorem 1]{MVJordan} and deduce from it the classification of the continuous endomorphisms of
$(\PP_2^1, \boxdot)$.

From Theorem~\ref{aut} of the present paper,  using the bicontinuous isomorphisms mentioned above (but in the opposite direction), we  infer

\begin{Cor} Every  automorphism of $(\D, \odot)$ or  $(\PP_2^1, \boxdot)$ is continuous.
\end{Cor}

Turning to arbitrary dimension, we have

\begin{Th}\label{end}
 For $n\ge 2$, continuous endomorphisms of the Einstein gyrogroup $(\B^n,
  \oplus)$ are precisely the restrictions to $\B^n$ of orthogonal
  transformations of $\R^n$ and the map that sends everything to 0.
\end{Th}

For $n=3$, this recovers the classifications of continuous endomorphisms of $(\B^3,\oplus)$, $(\D, \odot)$ and
$(\PP_2^1, \boxdot)$ given by Moln\'ar and Virosztek in \cite{MV}.

\begin{proof} It is clear that orthogonal transformations and the identically
  zero map are continuous endomorphisms.

Conversely, let $f$ be an arbitrary continuous endomorphism.

If $f$ is injective, then it is an open map, so its image contains a
neighbourhood of 0. But this neighbourhood generates $\B^n$ under $\oplus$, and
the image of $f$ is closed under $\oplus$, so $f$ must be surjective, i.e., $f$
is an automorphism. The claim now follows from Theorem~\ref{aut}.

If $f$ is not injective, then we have a  pair $u\ne v$ with
$f(u)=f(v)=f(u\oplus ((-u)\oplus v))
=f(u)\oplus f((-u)\oplus v)$. Let $x=(-u)\oplus v$, then $f(x)=0$ but
$x\ne 0$.  The diameter $L$ passing through $x$ is a subgroup isomorphic to
$(\R,+)$. We may choose an isomorphism such that $x$ corresponds to 1. It is
easy to see that $f(y)=0$ for every point of the diameter $L$ that corresponds to
a rational number. But then, by continuity, $f(y)=0$ for all $y$ on the
diameter $L$. Thus, $f$ is constant on sets of the form $a\oplus L$ and
$L\oplus b$. The former sets are lines in hyperbolic $n$-space, i.e., chords
of the ball $\B^n$. The chord $a\oplus L$ passes through $a$ and is parallel to $L$ if $a$ is orthogonal to $L$.  The latter
sets, when $b\notin L$,  are hypercycles in hyperbolic $n$-space, or half-ellipses in $\B^n$.
The half-ellipse  $L\oplus b$ connects the two ends of its major axis $L$ and passes
through $b$. It follows that $f$ is constant on any two-dimensional open half-disk
whose  boundary diameter is $L$. By continuity, $f=0$ everywhere.
\end{proof}

\section*
{Acknowledgements}  I am grateful to Lajos Moln\'ar and D\'aniel Virosztek
 for useful  conversations.

\end{document}